\documentclass[12pt, a4paper]{article}

\usepackage{amsfonts}
\usepackage{amssymb}
\usepackage{amscd}
\usepackage{mathrsfs}
\usepackage{amsmath,amssymb, amsthm}
\usepackage{graphicx}
\usepackage{color}
\usepackage{authblk}
\usepackage{indentfirst,latexsym,amssymb}
\usepackage[bookmarks=false]{hyperref}

\newtheorem{theorem}{Theorem}[section]

\newtheorem{lem}[theorem]{Lemma}

\newtheorem{coro}[theorem]{Corollary}

\theoremstyle{definition}

\newtheorem{rem}[theorem]{Remark}

\def\Cay{\hbox{\rm Cay}}

\def\ZZZ{\mathbb{Z}}

\def\Ga{\Gamma}

\usepackage{enumerate}
\long\def\delete#1{}

\usepackage{xcolor}
\usepackage[normalem]{ulem}

\input amssym.def
\input amssym.tex

\newcommand{\be}{\begin{equation}}
\newcommand{\ee}{\end{equation}}
\newcommand{\bea}{\begin{eqnarray}}
\newcommand{\eea}{\end{eqnarray}}
\newcommand{\bean}{\begin{eqnarray*}}
\newcommand{\eean}{\end{eqnarray*}}

\title{On subgroup perfect codes in Cayley graphs}
\author[a]{Junyang Zhang}
\author[b]{Sanming Zhou}
\affil[a]{{\small School of Mathematical Sciences, Chongqing Normal University, Chongqing 401331, P. R. China}}
\affil[b]{{\small School of Mathematics and Statistics, The University of Melbourne, Parkville, VIC 3010, Australia}}

\date{}

\begin{document}

\openup 0.5\jot
\maketitle

\renewcommand{\thefootnote}{\empty}
\footnotetext{E-mail addresses: jyzhang@cqnu.edu.cn (Junyang Zhang), sanming@unimelb.edu.au (Sanming Zhou)}

\begin{abstract}
A perfect code in a graph $\Ga = (V, E)$ is a subset $C$ of $V$ such that no two vertices in $C$ are adjacent and every vertex in $V \setminus C$ is adjacent to exactly one vertex in $C$. A subgroup $H$ of a group $G$ is called a subgroup perfect code of $G$ if there exists a Cayley graph of $G$ which admits $H$ as a perfect code. Equivalently, $H$ is a subgroup perfect code of $G$ if there exists an inverse-closed subset $A$ of $G$ containing the identity element such that $(A, H)$ is a tiling of $G$ in the sense that every element of $G$ can be uniquely expressed as the product of an element of $A$ and an element of $H$. In this paper we obtain multiple results on subgroup perfect codes of finite groups, including a few necessary and sufficient conditions for a subgroup of a finite group to be a subgroup perfect code, a few results involving $2$-subgroups in the study of subgroup perfect codes, and several results on subgroup perfect codes of metabelian groups, generalized dihedral groups, nilpotent groups and $2$-groups.


\end{abstract}

\section*{Preface}

This is the corrected version of our paper under the same title published in \emph{European J. Combin.} 91 (2021) 103228 (see \url{https://doi.org/10.1016/j.ejc.2020.103228} or \url{arXiv:2006.11104v1}). After the paper was published, an error was found in the proof of Theorem 3.1. (The claim that ``$xH\neq yH$ for distinct elements $x, y \in X$'' is invalid as $z^{-1}H$ may be equal to $wH$ for some $w\in T$.) We are grateful to Dr. Kai Yuan for bringing this error to our attention.

In this corrected version, we will revise Theorem 3.1 and its proof. Since Theorem 3.2 relies on Theorem 3.1, its statement and proof need to be revised as well. It turns out that the subgroup involved in Theorem 3.2 is required to be a $2$-group or have odd order or odd index, and the same condition is required in Corollaries 3.3 and 3.4 as well. The proofs of Theorems 3.7, 4.2 and 4.3 in the published version need to be modified due to their reliance on Corollaries 3.3 and 3.4, but the statements of these results remain true and hence are not changed. We will give new proofs of these three theorems using our revised Theorem 3.1.  

\section{Introduction}

All groups considered in the paper are finite, and all graphs considered are finite and undirected. Group-theoretic terminology and notation used in the paper are standard and can be found in, for example, \cite{KS2004}.

Let $\Gamma$ be a graph with vertex set $V(\Gamma)$ and edge set $E(\Gamma)$, and let $e$ be a positive integer. A subset $C$ of $V(\Gamma)$ is called \cite{Big, Kr86} a \emph{perfect $e$-error-correcting code} (or \emph{perfect $e$-code} for short) in $\Gamma$ if every vertex of $\Gamma$ is at distance no more than $e$ to exactly one vertex in $C$, where the \emph{distance} in $\Ga$ between two vertices is the length of a shortest path between the two vertices or $\infty$ if there is no path in $\Gamma$ joining them. A perfect $1$-code is usually called a {\em perfect code}. Equivalently, a subset $C$ of $V(\Gamma)$ is a perfect code in $\Gamma$ if $C$ is an independent set of $\Gamma$ and every vertex in $V(\Gamma) \setminus C$ is adjacent to exactly one vertex in $C$. A perfect code in a graph is also called an efficient dominating set \cite{DeS} or independent perfect dominating set \cite{Le} of the graph. As a convention, when $\Ga$ is an empty graph (that is, $E(\Ga) = \emptyset$), we treat $V(\Ga)$ as a perfect code in $\Ga$.

The notion of perfect $e$-codes in graphs originated \cite{Big, Kr86} from coding theory. In the case when $\Ga$ is the Hamming graph $H(n, q)$, the Hamming distance between words of length $n$ over an alphabet of size $q \ge 2$ is precisely the graph distance in $\Ga$, and therefore perfect $e$-codes in $\Ga$ are exactly those in the classical setting \cite{MS77} under the Hamming metric. Similarly, when $\Ga$ is the Cartesian product $L(n, q)$ of $n$ copies of the cycle of length $q \ge 3$, the Lee distance \cite{HK18} between words of length $n$ over an alphabet of size $q \ge 3$ is exactly the graph distance in $L(n, q)$, and hence perfect $e$-codes in $\Ga$ are precisely those under the Lee metric.

It is well known that Hamming graphs are distance-transitive \cite{BCN89}. This motivated Biggs \cite{Big} to study perfect $e$-codes in distance-transitive graphs as a generalization of perfect $e$-codes under the Hamming metric. In \cite{Big}, among other things Biggs generalized the celebrated Lloyd's Theorem \cite{Lens} for Hamming graphs to all distance-transitive graphs. Since the seminal works of Biggs \cite{Big} and Delsarte \cite{Del}, an extensive body of research has been devoted to perfect codes in distance-transitive graphs and, in general, in distance-regular graphs and association schemes \cite{BCN89}.
See, for example, \cite{B77, B77a, HS} and the survey papers \cite{Heden1, Va75}.

\medskip
\textbf{Perfect codes in Cayley graphs.}~
As mentioned in \cite{HXZ18}, perfect $e$-codes in Cayley graphs are another generalization of perfect $e$-codes under the Hamming or Lee metric. This is so because both $H(n, q)$ and $L(n, q)$ are Cayley graphs of the additive group $\ZZZ_q^n$. In fact, they are Cayley graphs of $\ZZZ_q^n$ with connection sets $S_H$ and $S_L$, respectively, where $S_H$ consists of all elements of $\ZZZ_q^n$ with precisely one nonzero coordinate, and $S_L$ consists of all elements of $\ZZZ_q^n$ such that exactly one coordinate is $\pm 1\mod q$ and all other coordinates are zero. In general, for a group $G$ with identity element $1$ and an \emph{inverse-closed} subset $S$ of $G \setminus \{1\}$ (that is, $S^{-1}:=\{x^{-1}: x \in S\} = S$), the {\em Cayley graph} ${\rm Cay}(G, S)$ of $G$ with \emph{connection set} $S$ is defined as the graph with vertex set $G$ such that two elements $x, y$ of $G$ are adjacent if and only if $yx^{-1}\in S$. For convenience, we allow $S$ to be $\emptyset$ or $G \setminus \{1\}$, in which case ${\rm Cay}(G, S)$ is the empty graph or complete graph with vertex set $G$, respectively.

Apart from being significant generalizations of perfect codes in the classical setting, perfect codes in Cayley graphs are also of considerable importance for factorizations and tilings of groups. A \emph{factorization} \cite{SS} of a group $G$ (into two factors) is an ordered pair of subsets $(A, B)$ of $G$ such that every element of $G$ can be uniquely written as $ab$ with $a \in A$ and $b \in B$. (Note that, unlike factorizations in group theory, here we do not need $A$ and $B$ to be subgroups of $G$, but we require a unique representation $ab$ of each element of $G$.) A factorization $(A, B)$ of $G$ such that $1 \in A \cap B$ is called a \emph{tiling} \cite{Dinitz06} or a normed factorization of $G$. Beginning with Haj\'{o}s \cite{Hajos} in his proof of a well-known conjecture of Minkowski, there is a long history of studying factorizations and tilings of Abelian groups; see, for example, \cite{Dinitz06, Szabo06, SS} for related results and background information. It is readily seen (see, for example, \cite[Lemma 2.10]{HXZ18}) that $(A, B)$ is a tiling of $G$ with $A$ inverse-closed if and only if $B$ is a perfect code of $\Cay(G, A \setminus \{1\})$ with $1 \in B$.

\medskip
\textbf{Subgroup perfect codes.}~
In recent years perfect codes in Cayley graphs have received considerable attention; see \cite[Section 1]{HXZ18} for a brief survey and \cite{DSLW16, FHZ, Ta13, Z15} for a few recent papers. In particular, perfect codes in Cayley graphs which are subgroups of the underlying groups are especially interesting since they are generalizations of perfect linear codes \cite{Va75} in the classical setting. Another interesting avenue of research is to study when a given subset of a group is a perfect code in some Cayley graph of the group. In this regard the following concepts were introduced by Huang et al. in \cite{HXZ18}: A subset $C$ of a group $G$ is called a {\em perfect code} of $G$ if there exists a Cayley graph of $G$ which admits $C$ as a perfect code; a perfect code of $G$ which is also a subgroup of $G$ is called \cite{HXZ18} a {\em subgroup perfect code} of $G$. As a convention, any group $G$ is considered as a perfect code of itself since $G$ is a perfect code in the empty Cayley graph $\Cay(G, \emptyset)$. The trivial subgroup $\{1\}$ is also a perfect code of $G$ since it is a perfect code in the complete graph $\Cay(G, G \setminus \{1\})$. From a tiling point of view, the problem of determining whether a subgroup $H$ of $G$ is a subgroup perfect code of $G$ is the one of determining whether there exists an inverse-closed subset $A$ of $G$ with $1 \in A$ such that $(A, H)$ is a tiling of $G$. Since $H$ is a subgroup of $G$, such a tiling $(A, H)$ takes the role of lattice tilings of $\ZZZ^n$, say, with each tile a copy of $A$. Requiring $A$ to be inverse-closed ensures that the underlying Cayley graph $\Cay(G, A \setminus \{1\})$ is undirected. This requirement also makes our problem interesting and challenging. In fact, without this condition any subgroup $H$ of $G$ would be a ``perfect code" in the (not necessarily undirected) Cayley graph $\Cay(G, T \setminus \{1\})$, where $T$ is any left transversal of $H$ in $G$ which contains $1$.

In \cite{HXZ18}, Huang et al. obtained among others a necessary and sufficient condition for a normal subgroup of a group to be a perfect code (see \cite[Theorem 2.2]{HXZ18} which is presented below as Lemma \ref{HXZ}) and determined all subgroup perfect codes of all dihedral groups and some Abelian groups. In \cite[Theorem 1.1]{MWWZ19}, Ma et al. proved that a group has the property that every proper subgroup is a perfect code if and only if the group has no elements of order $4$. In particular, every group of odd order has this property, and an Abelian group has this property if and only if it is isomorphic to the product of an elementary Abelian $2$-group and an Abelian group of odd order. In \cite[Theorem 1.5]{MWWZ19}, it was proved that a proper subgroup $H$ of an Abelian group $G$ is a perfect code of $G$ if and only if the Sylow $2$-subgroup of $H$ is a perfect code of the Sylow $2$-subgroup of $G$. All subgroup perfect codes of generalized quaternion groups have also been determined in \cite[Theorem 1.7]{MWWZ19}.

\medskip
\textbf{Main results.}~
In this paper we first prove several general results on subgroup perfect codes in Section \ref{sec:gen}, including (partial) generalizations of Theorem 2.2(a) and Corollary 2.3(a) in \cite{HXZ18}. We obtain some necessary and sufficient conditions for a subgroup to be a subgroup perfect codes (see Theorems \ref{basic} and \ref{normal} and Corollary \ref{equivalent}), and prove that any subgroup with odd order or odd index is a perfect code of the group (see Theorem \ref{odd}). We further prove that the property of being a subgroup perfect code is carried over when taking quotients by normal subgroups (see Theorem \ref{quotient}) and that under certain conditions a subgroup of a subgroup perfect code is also a subgroup perfect code (see Theorem \ref{commutator}).

In Section \ref{sec:2subgp}, we investigate the role played by 2-subgroups and Sylow 2-subgroups in the study of subgroup perfect codes and obtain a few results in this line of research (see Theorems \ref{sylow}--\ref{QK}). In Section \ref{sec:classes}, we obtain a necessary and sufficient condition for a subgroup of a generalized dihedral group or a nilpotent group to be a perfect code of the group; see Theorems \ref{gen-dih} and \ref{nilp}, the former being a generalization of \cite[Theorem 2.11(a)]{HXZ18}. We also obtain a necessary and sufficient condition for a normal subgroup of a metabelian group to be a subgroup perfect code (see Theorem \ref{ma}). Finally, we prove that for any $2$-group either every element not in its Frattini subgroup generates a (cyclic) subgroup perfect code or the $2$-group admits a generalized quaternion subgroup as a perfect code (see Theorem \ref{2gp}).

A \emph{total perfect code} \cite{Zhou2016} in a graph $\Ga$, also known as an efficient open dominating set \cite{HHS} of $\Ga$, is a subset $C$ of $V(\Ga)$ such that every vertex of $\Ga$ is adjacent to exactly one vertex in $C$. A subgroup of a group $G$ which is a total perfect code in some Cayley graph of $G$ is called \cite{HXZ18} a {\em subgroup total perfect code} of $G$. Using the well-known result that any group of even order contains at least one involution, one can verify that a subgroup perfect code of a group is a total perfect code of the group if and only if it is of even order. Based on this observation one can see that all results in this paper are also true if we replace the phrases ``perfect code" and ``subgroup perfect code" by ``total perfect code" and ``subgroup total perfect code", respectively, and add the condition that the subgroup under consideration is of even order.

\section{Lemmas}

This short section containing four lemmas is a preparation for later sections. Given a group $G$ and a subgroup $H$ of $G$, we call a subset of $G$ a \emph{Cayley transversal} of $H$ in $G$ if it is a right transversal of $H$ in $G$ which is closed under taking inverse elements. Note that we can replace ``right transversal" by ``left transversal" in this definition as an inverse-closed subset of $G$ is a right transversal of $H$ in $G$ if and only if it is a left transversal of $H$ in $G$. The following lemma follows immediately from the definition of a subgroup perfect code.

\begin{lem}
\label{Ct}
Let $G$ be a group and $H$ a subgroup of $G$. Then $H$ is a perfect code of $G$ if and only if it has a Cayley transversal in $G$.
\end{lem}

\begin{lem}
\label{sub}
Let $G$ be a group and $H$ a subgroup of $G$. Then $H$ is a perfect code of $G$ if and only if it is a perfect code of any subgroup of $G$ which contains $H$.
\end{lem}

\begin{proof}
It suffices to prove the necessity. Let $H$ be a subgroup perfect code of $G$ and $K$ an arbitrary subgroup of $G$ which contains $H$. By Lemma \ref{Ct}, $H$ has a Cayley transversal in $G$, say, $X$. So $(X, H)$ is a factorization of $G$. Set $Y=X \cap K$. Since both $X$ and $K$ are inverse-closed, so is $Y$. We have $K=G\cap K = XH \cap K = HX \cap K= H (X \cap K) = HY = YH$. Hence $(Y, H)$ is a factorization of $K$ and $Y$ is Cayley transversal of $H$ in $K$. Therefore, by Lemma \ref{Ct}, $H$ is a perfect code of $K$.
\end{proof}

The following lemma can be easily proved.

\begin{lem}
\label{conjugate}
Let $G$ be a group and $H$ a subgroup of $G$. If $H$ is a perfect code of $G$, then for any $g \in G$, $g^{-1}Hg$ is a perfect code of $G$. More specifically, if $H$ is a perfect code in $\Cay(G, S)$ for some connection set $S$ of $G$, then $g^{-1}Hg$ is a perfect code in $\Cay(G, g^{-1}Sg)$.
\end{lem}

The next lemma is taken from \cite{HXZ18}.

\begin{lem}
\label{HXZ}
(\cite[Theorem 2.2]{HXZ18}) Let $G$ be a group and $H$ a normal subgroup of $G$. Then $H$ is a perfect code of $G$ if and only if for all $x\in G$, $x^{2}\in H$ implies $(xh)^{2}=1$ for some $h\in H$.
\end{lem}

\section{Some general results on subgroup perfect codes}
\label{sec:gen}

Our first result, Theorem \ref{basic} below, is a basic tool for proving subsequent results in the rest of this paper. Recall that an element of a group is called a \emph{$2$-element} if its order is a power of $2$. In particular, the identity element is treated as a $2$-element. As usual, denote by $|G|$ the order of a group $G$, $|G:H|$ the index in $G$ of a subgroup $H$ of $G$, and $|X|$ the cardinality of a set $X$. 

\begin{theorem}
\label{basic}
Let $G$ be a group and $H$ a subgroup of $G$. Then $H$ is a perfect code of $G$ if and only if for any $g\in G\setminus H$ either the left coset $gH$ contains an involution or the integer $|H\{g,g^{-1}\}H|/|H|$ is even. In particular, if $H$ is not a perfect code of $G$, then there exists a $2$-element $x\in G\setminus H$ such that $x^{2}\in H$, $|H:H\cap xHx^{-1}|$ is odd, and $xH$ contains no involution.
\end{theorem}

\begin{proof}
We prove the necessity first. Suppose that $H$ is a perfect code of $G$. By Lemma \ref{Ct}, $H$ has a Cayley transversal $T$ in $G$. Consider an arbitrary element $g\in G\setminus H$. Obviously,  $H\{g,g^{-1}\}H$ is the union of some left cosets of $H$ in $G$. So $m = |H\{g,g^{-1}\}H|/|H|$ is a positive integer, and $H\{g,g^{-1}\}H$ is the union of $m$ distinct left cosets of $H$ in $G$, say, $g_{1}H, \ldots, g_{m}H$, where $g_{1},\ldots,g_{m}\in G$.
It suffices to show that $gH$ contains an involution when $m$ is odd. In fact, since $T$ is a left transversal of $H$ in $G$, for each $1\leq i\leq m$, there exists $x_i \in G$ such that $T\cap g_{i}H = \{x_{i}\}$. Set $X=\{x_1,x_2,\ldots,x_m\}$. Since $T$ is inverse-closed, we have
$X^{-1}=(T\cap H\{g,g^{-1}\}H)^{-1}=T^{-1}\cap (H\{g,g^{-1}\}H)^{-1}=T\cap H\{g,g^{-1}\}H=X$.
Therefore, if $m$ is odd, then $X$ contains at least one involution, say, $x$. Since $x\in H\{g,g^{-1}\}H$ and $x^{-1}=x$, we have $x\in HgH$. So $x=h_{1}gh_{2}$ for some $h_{1},h_{2}\in H$. Since $g=h_{1}^{-1}xh_{2}^{-1}$, it follows that $h_{1}^{-1}xh_{1}$ is an involution contained in $gH$.

Now we prove the sufficiency. Assume that for any $g\in G\setminus H$ either $gH$ contains an involution or $m = |H\{g,g^{-1}\}H|/|H|$ is even. Take a subset $T$ of $G$ with maximum cardinality such that $1\in T$, $T^{-1}=T$, $HTH=TH$ and $xH\neq yH$ for all pairs of distinct elements $x,y\in T$. (The existence of $T$ follows from the fact that there are subsets of $G$, say, $\{1\}$, with all these properties.)
By Lemma \ref{Ct}, it suffices to show that $T$ is a left transversal of $H$ in $G$. Suppose otherwise. Then $G \setminus TH \neq\emptyset$ and therefore we can take an element $g\in G \setminus TH$. Since $1\in T$, we have $g \in G\setminus H$. Since $HTH=TH$ and $g\notin TH$, we have $HgH\cap TH=\emptyset$. Furthermore, since $T=T^{-1}$, we have $Hg^{-1}H\cap TH=Hg^{-1}H\cap HTH=(HgH\cap HTH)^{-1}=\emptyset$.
Set $|H:H\cap gHg^{-1}|=\ell$. Since $H\cap g^{-1}Hg=g^{-1}(H\cap gHg^{-1})g$, we have $|H:H\cap g^{-1}Hg|=\ell$. It is straightforward to verify that $h_{1}gH=h_{2}gH$ if and only if $h_{1}(H\cap gHg^{-1})=h_{2}(H\cap gHg^{-1})$ for any pair of elements $h_{1},h_{2}\in H$. Therefore, $HgH$ is the union of $\ell$ distinct left cosets of $H$ in $G$. So we can express $HgH$ as the union of $\ell$ distinct left cosets $x_{1}gH,\ldots,x_{\ell}gH$ for some $x_{1},\ldots,x_{\ell} \in H$. Similarly, we can write $Hg^{-1}H$ as the union of $\ell$ distinct left cosets $y_{1}g^{-1}H,\ldots,y_{\ell}g^{-1}H$ for some $y_{1},\ldots,y_{\ell}\in H$.  
If $gH$ contains an involution $z$, then we set 
$$
X=T\cup\{x_{1}zx_{1}^{-1},\ldots,x_{\ell}zx_{\ell}^{-1}\}.
$$
If $HgH\cap Hg^{-1}H=\emptyset$, then we set
\begin{equation*}
X=T\cup\{x_{1}gy_{1}^{-1},\ldots,x_{\ell}gy_{\ell}^{-1},y_{1}g^{-1}x_{1}^{-1},\ldots,y_{\ell}g^{-1}x_{\ell}^{-1}\}.
\end{equation*}
If $gH$ contains no involution and $HgH\cap Hg^{-1}H\neq\emptyset$, then $g^{-1}H=hgH$ for some $h\in H$, and by our assumption, $H\{g,g^{-1}\}H$ is the union of $m$ distinct left cosets of $H$ in $G$. Hence $HgH=Hg^{-1}H=H\{g,g^{-1}\}H$ and $\ell = m$. Therefore, $\ell$ is even and we can set
\begin{equation*}
X=T\cup \left\{x_{1}ghx_{\frac{\ell}{2}+1}^{-1},\ldots,x_{\frac{\ell}{2}}ghx_{\ell}^{-1},
x_{\frac{\ell}{2}+1}h^{-1}g^{-1}x_{1}^{-1},\ldots,x_{\ell}h^{-1}g^{-1}x_{\frac{\ell}{2}}^{-1}\right\}.
\end{equation*}
In each case above, we have defined a subset $X$ of $G$ which contains $T$ as a proper subset. It can be verified that $X^{-1} = X$, $HXH=XH$ and $xH\neq yH$ for any pair of distinct elements $x,y\in X$, but this contradicts the maximality of $T$. This contradiction shows that $T$ must be a left transversal of $H$ in $G$. The sufficiency then follows from Lemma \ref{Ct}.

It remains to prove the last statement in the theorem. Suppose that $H$ is not a perfect code of $G$. Then by what we have proved above there exists $g\in G\setminus H$ such that $Hg^{-1}H=HgH$, $|H:H\cap gHg^{-1}|$ is odd, and $gH$ contains no involution. Since $Hg^{-1}H=HgH$, we have $Hg^{-1}=Hgh$ for some $h\in H$. Set $y=gh$. Then $y\in G\setminus H$ and $y^2\in H$. Let $s$ be the largest odd divisor of the order of $y$. Set $x=y^{s}$. Then $x^2 \in H$. Since $s-1$ is even and $y^2\in H$, we have $y^{s-1}\in H$. Hence $xH=yH$ and $xHx^{-1}=yHy^{-1}=ghHh^{-1}g^{-1}=gHg^{-1}$. Note that $x \not \in H$ as $y \not \in H$. Therefore, $x \in G\setminus H$ is a $2$-element such that $x^{2}\in H$,  $|H:H\cap xHx^{-1}|$ is odd, and $xH$ contains no involution.
\end{proof}

The next result ensures that for a group with odd order testing whether a subgroup is a perfect code can be reduced to testing whether it is a perfect code of its normalizer in the group. This implies that, in theory, the study of subgroup perfect codes of groups with odd orders can be reduced to the study of ``normal subgroup perfect codes". 
   
\begin{theorem}
\label{normal}
Let $G$ be a group and $H$ a subgroup of $G$. Suppose that either $H$ is a $2$-group or at least one of $|H|$ and $|G:H|$ is odd. Then $H$ is a perfect code of $G$ if and only if $H$ is a perfect code of $N_{G}(H)$.
\end{theorem}

\begin{proof}
The necessity follows from Lemma \ref{sub}. To prove the sufficiency, we assume that $H$ is a perfect code of $N_{G}(H)$. By way of contradiction, suppose that $H$ is not a perfect code of $G$. By Theorem \ref{basic}, there exists a $2$-element $x\in G\setminus H$ such that $x^2\in H$, $|H:H\cap xHx^{-1}|$ is odd, and $xH$ contains no involution. Since $H$ is a perfect code of $N_{G}(H)$, we have $x\notin N_{G}(H)$ by Lemma \ref{HXZ}. Set $L=H\cap xHx^{-1}$. Then $L$ is a proper subgroup of $H$. Since $|H:L|$ is odd, $H$ cannot be a $2$-group. Since $xH$ contains no involution, $x$ is not an involution. Therefore, $x^2$ is a non-identity $2$-element. Since $x^2\in H$, it follows that $H$ is of even order. Since $xLx^{-1}= x(H\cap xHx^{-1})x^{-1} = xHx^{-1}\cap x^2Hx^{-2} = xHx^{-1}\cap H = L$, we have $x\in N_{G}(L)$. It follows that $|N_{G}(L):L|$ is even. Since $|N_{G}(L):L|$ is a divisor of $|G:L|$, it follows that $|G:L|$ is even. Since $|G:L|=|G:H||H:L|$ and $|H:L|$ is odd, we conclude that $|G:H|$ is even. Now we have proved that $H$ is not a $2$-group and both $|H|$ and $|G:H|$ are even. This contradicts our assumption and therefore $H$ must be a perfect code of $G$.
\end{proof}

Combining Lemma \ref{HXZ} and Theorem \ref{normal}, we obtain the following result.  

\begin{coro}\label{equivalent}
Let $G$ be a group and $H$ a subgroup of $G$. Suppose that either $H$ is a $2$-group or at least one of $|H|$ and $|G:H|$ is odd. Then $H$ is a perfect code of $G$ if and only if for any $x\in N_{G}(H)$, $x^{2}\in H$ implies $(xh)^{2}=1$ for some $h\in H$.
\end{coro}

\begin{coro}\label{equivalent2}
Let $G$ be a group and $H$ a subgroup of $G$. Suppose that either $H$ is a $2$-group or at least one of $|H|$ and $|G:H|$ is odd. Then $H$ is a perfect code of $G$ if and only if for any $2$-element $x\in N_{G}(H)$, $x^{2}\in H$ implies $(xh)^{2}=1$ for some $h\in H$.
\end{coro}

\begin{proof}
The necessity follows from Corollary \ref{equivalent} immediately. To prove the sufficiency, we assume that for any $2$-element $x\in N_{G}(H)$, $x^{2}\in H$ implies $(xh)^{2}=1$ for some
$h\in H$. Let $y$ be an arbitrary element of $N_{G}(H)$, and let $2^{k}s$ be the order of $y$, where $k$ is a nonnegative integer and $s \ge 1$ is an odd integer. Set $x=y^{s}$. Then $x$ is a $2$-element in $N_{G}(H)$. If $y^{2}\in H$, then $x^{2}=y^{2s}\in H$ and therefore, by our assumption,  there exists $h\in H$ such that $(xh)^{2}=1$. Since $s$ is an odd integer and $y^{2}, h\in H$, we have $y^{s-1}h\in H$. Since $(yy^{s-1}h)^{2} = (xh)^{2} = 1$, it follows from Corollary \ref{equivalent} that $H$ is a perfect code of $G$.
\end{proof}

The following is a generalization of \cite[Corollary 2.3(a)]{HXZ18} (which in turn implies the ``if" part in \cite[Theorem 3.6]{Ta13}), where the same statement was proved under the additional condition that the subgroup involved is normal. We show that the same result is true without this additional condition. In particular, we recover the known result (see \cite[Corollary 1.2]{MWWZ19}) that in any group of odd order every proper subgroup is a perfect code.

\begin{theorem}
\label{odd}
Let $G$ be a group and $H$ a subgroup of $G$. If either the order of $H$ is odd or the index of $H$ in $G$ is odd, then $H$ is a perfect code of $G$.
\end{theorem}

\begin{proof} 
Suppose first that $|H|$ is odd. Consider an arbitrary element $x\in N_{G}(H)$ with $x^{2}\in H$. Assume that the order of $x^{2}$ is $m$. Then $m$ is an odd number. Since $x^{2}\in H$, it follows that $x^{m-1} \in H$. Since $(xx^{m-1})^{2}=1$, by Corollary \ref{equivalent} we obtain that $H$ is a perfect code of $G$.

Now suppose that $|G:H|$ is odd. Then $|N_{G}(H):H|$ is odd. Thus, for any $x\in N_{G}(H)$, $x^{2}\in H$ implies $(xh)^{2}=1$, where $h=x^{-1}\in H$. Therefore, by Corollary \ref{equivalent}, $H$ is a perfect code of $G$.
\end{proof}

\begin{rem}
In \cite[Theorem 3.6]{Ta13}, it was proved that a proper subgroup of a cyclic group is a perfect code if and only if it has an odd order or odd index. So for cyclic groups the converse of the statement in Theorem \ref{odd} is true.
Since cyclic 2-groups and generalized quaternion 2-groups have no nontrivial subgroup perfect code (see \cite[Theorems 1.6 and 1.7]{MWWZ19}), it follows from Theorem \ref{QK} (see the next section) that the converse of the statement in Theorem \ref{odd} is also true when $G$ is the direct product of a cyclic or generalized quaternion 2-group and a group of odd order. At present we do not know any other class of groups for which the converse statement in Theorem \ref{odd} is true. On the other hand, the famous binary Hamming codes show that the converse statement in Theorem \ref{odd} fails for elementary Abelian $2$-groups.
\end{rem}

The next result shows that the property of being a subgroup perfect code is inherited by quotient subgroups, and that the converse is also true when the normal subgroup involved is a perfect code.

\begin{theorem}
\label{quotient}
Let $G$ be a group, $N$ a normal subgroup of $G$, and $H$ a subgroup of $G$ which contains $N$. Then the following hold:
\begin{enumerate}[\rm (a)]
  \item if $H$ is a perfect code of $G$, then $H/N$ is a perfect code of $G/N$;
  \item if $N$ and $H/N$ are perfect codes of $G$ and $G/N$, respectively, then $H$ is a perfect code of $G$.
\end{enumerate}
\end{theorem}

\begin{proof}
(a) Suppose that $H$ is a perfect code of $G$. By Lemma \ref{Ct}, there exists a Cayley transversal $T$ of $H$ in $G$. By the definition of a Cayley transversal, we then have $T^{-1}=T$, $G=TH$ and $xH\neq yH$ for any pair of distinct elements $x,y\in T$. Write $T/N = \{xN: x\in T\}$. Then $(T/N)^{-1}=T/N$ and $G/N=(T/N)(H/N)$. Moreover, since $N$ is contained in $H$, we have $(xN)H/N\neq (yN)H/N$ for any pair of distinct elements $xN,yN \in T/N$. Therefore, $T/N$ is a Cayley transversal of $H/N$ in $G/N$. By Lemma \ref{Ct}, $H/N$ is a perfect code of $G/N$.

(b) Suppose to the contrary that $H$ is not a perfect code of $G$. By Theorem \ref{basic}, there exists a $2$-element $x\in G\setminus H$ such that $x^{2}\in H$, $|H:H\cap xHx^{-1}|$ is odd, and $xH$ contains no involution. Since $N$ is a normal subgroup of $G$ contained in $H$, $xN$ is a $2$-element in $(G/N)\setminus(H/N)$ and $(xN)^{2}\in H/N$. Moreover, $|H/N:(H/N)\cap (xN)(H/N)(xN)^{-1}|$ is odd as it is equal to $|H/N:(H\cap xHx^{-1})/N| = |H:H\cap xHx^{-1}|$. Since $H/N$ is a perfect code of $G/N$, by Theorem \ref{basic}, $(xN)(H/N)$ contains an involution. Therefore, $(xNaN)^{2}=N$ for some $aN\in H/N$. That is, $(xa)^{2}N=N$ and hence $(xa)^{2}\in N$. Since $N$ is a perfect code and a normal subgroup of $G$, by Lemma \ref{HXZ}, there exists $b\in N$ such that $(xab)^{2}=1$. Note that $ab\in H$. However, $xH$ contains no involution, a contradiction.
\end{proof}

As usual, for any subsets $A$ and $B$ of a group $G$, we use $[A,B]$ to denote the subgroup of $G$ generated by all commutators $[a,b] = a^{-1}b^{-1}ab$ with $a\in A$ and $b\in B$. Our last result in this section gives two sufficient conditions for a subgroup of a subgroup perfect code to be a subgroup perfect code. Note that in this result $H$ is necessarily normal in $G$, for otherwise $H$ cannot contain $[G, H]$ and therefore no subgroup of $H$ can contain $[G, H]$.

\begin{theorem}
\label{commutator}
Let $G$ be a group, $H$ a normal subgroup of $G$ and $K$ a subgroup of $H$ which contains $[G,H]$. If $H$ is a perfect code of $G$, then $K$ is a perfect code of $G$ provided that one of the following conditions holds:
\begin{enumerate}[\rm (a)]
  \item $[G,H]$ is a perfect code of $G$ and $K$ is a perfect code of $H$;
  \item $K$ is of odd index in $H$.
\end{enumerate}
\end{theorem}

\begin{proof}
Let $G$ be a group, $H$ a normal subgroup of $G$ and $K$ a subgroup of $H$ which contains $[G,H]$.

(a) Assume that $[G,H]$ is a perfect code of $G$ and $K$ is a perfect code of $H$. Since $g^{-1}ag=[g,a^{-1}]a\in K$ for all $g\in G$ and $a\in K$, $K$ is a normal subgroup of $G$. Consider an arbitrary element $x\in G$ with $x^{2}\in K$. Then $x^{2}\in H$. Since $H$ is a normal subgroup and a perfect code of $G$, by Lemma \ref{HXZ} we have $(xh)^{2}=1$ for some $h\in H$. Since $1=(xh)^{2}=x^{2}[x,h^{-1}]h^{2}$, we get $h^{2}=[h^{-1},x]x^{-2}\in K$. Since $K$ is normal in $G$, we get that $K$ is normal in $H$. Since $K$ is a perfect code of $H$, by Lemma \ref{HXZ} we have $(hb)^{2}=1$ for some $b\in K$. Then $b^{-2}=h^{2}[h,b^{-1}]$ and therfore\begin{align*}
  (xb^{-1})^{2} & =  x^{2}[x,b]b^{-2}\\
  & = [x^{2},[x,b]][x,b]x^{2}b^{-2} \\
  & =  [x^{2},[x,b]][x,b]x^{2}h^{2}[h,b^{-1}]\\
  & = [x^{2},[x,b]][x,b][h^{-1},x][h,b^{-1}].
\end{align*}
It follows that $(xb^{-1})^{2}\in [G,H]$. Since $[G,H]$ is a normal subgroup and a perfect code of $G$, by Lemma \ref{HXZ} we have $(xb^{-1}c)^{2}=1$ for some $c\in [G,H]$. Since $K$ contains $[G,H]$, we have $c \in K$ and so $b^{-1}c\in K$. By Lemma \ref{HXZ}, $K$ is a perfect code of $H$.

(b) Assume that $K$ is of odd index in $H$. Since $H$ is a perfect code of $G$, by Lemma \ref{Ct}, $H$ has a Cayley transversal $X$ in $G$. Write $X=\{1\}\cup X_{0}\cup X_{1}\cup X_{1}^{-1}$, where $X_{0}$ consists of all involutions in $X$. Since the index of $K$ in $H$ is odd, by Theorem \ref{odd}, $K$ is a perfect code of $H$. Therefore, $K$ has a Cayley transversal $Y$ in $H$. Write $Y=\{1\}\cup Y_{1}\cup Y_{1}^{-1}$. Then
$$
XY=X\cup Y\cup X_{0}Y_{1}\cup X_{0}Y_{1}^{-1}\cup
X_{1}Y_{1}\cup X_{1}Y_{1}^{-1}\cup X_{1}^{-1}Y_{1}\cup X_{1}^{-1}Y_{1}^{-1}.
$$
Moreover, $XY$ is a transversal of $K$ in $G$. By Lemma \ref{Ct}, to complete the proof it suffices to construct a Cayley transversal of $K$ in $G$. Since $[G,H] \le K$, for any $x\in X$ and $y\in Y$, we have $xyx^{-1}y^{-1}\in K$ and so $x^{-1}y^{-1}K=y^{-1}x^{-1}K$.
Set
\begin{equation*}
T=X\cup Y\cup X_{0}Y_{1}\cup Y_{1}^{-1}X_{0}\cup
X_{1}Y_{1}\cup Y_{1}^{-1}X_{1}^{-1}\cup X_{1}Y_{1}^{-1}\cup Y_{1}X_{1}^{-1}.
\end{equation*}
Then $T$ is a Cayley transversal of $K$ in $G$ as required.
\end{proof}

An immediate corollary of Theorem \ref{commutator} is as follows.

\begin{coro}
Let $G$ be an Abelian group and $H$ a subgroup of $G$. If $H$ is a perfect code of $G$, then any subgroup perfect code of $H$ is also a perfect code of $G$.
\end{coro}

\section{$2$-subgroups}
\label{sec:2subgp}

A subgroup with order a power of $2$ is called a \emph{$2$-subgroup}. In this section we investigate the role played by $2$-subgroups in the study of subgroup perfect codes. The first result stated below ensures that testing whether a $2$-subgroup is a perfect code of a group can be reduced to testing whether it is a perfect code of the Sylow $2$-subgroups containing it.

\begin{theorem}
\label{sylow}
Let $G$ be a group.
\begin{enumerate}[\rm (a)]
\item Let $Q$ be a $2$-subgroup of $G$. Then $Q$ is a perfect code of $G$ if and only if it is a perfect code of every Sylow $2$-subgroup of $G$ which contains $Q$.
\item Let $P$ be a Sylow $2$-subgroup of $G$ and $Q$ a normal subgroup of $P$. Then $Q$ is a perfect code of $G$ if and only if it is a perfect code of $P$.
\end{enumerate}
\end{theorem}

\begin{proof}
(a) The necessity follows directly from Lemma \ref{sub}. Now we assume that $Q$ is a perfect code of every Sylow $2$-subgroup of $G$ which contains $Q$. Consider any $x\in N_{G}(Q)$ with $x^{2}\in Q$. There exists a Sylow $2$-subgroup $P$ of $G$ which contains $x$ and $Q$. Clearly, $x\in N_{P}(Q)$. By our assumption, $Q$ is a perfect code of $P$. Thus, by Corollary \ref{equivalent}, there exists $b\in Q$ such that $(xb)^{2}=1$. Therefore, by Corollary \ref{equivalent} again, $Q$ is a perfect code of $G$.

(b) By Theorem \ref{normal} and (i), $Q$ is a perfect code of $G$ if and only if it is a perfect code of every Sylow $2$-subgroup of $N_{G}(Q)$ which contains $Q$. Since $Q$ is normal in $P$, the Sylow $2$-subgroups of $N_{G}(Q)$ form the conjugacy class of $P$ in $N_{G}(Q)$. Therefore, the result follows from Lemma \ref{conjugate} directly.
\end{proof}

The next result asserts that a subgroup is a perfect code of a group if it has a Sylow $2$-subgroup which is a perfect code of the group.

\begin{theorem}
\label{ns}
Let $G$ be a group and $H$ a subgroup of $G$. If there exists a Sylow $2$-subgroup of $H$ which is a perfect code of $G$, then $H$ is a perfect code of $G$.
\end{theorem}

\begin{proof}
Let $g$ be an arbitrary element of $G\setminus H$ such that $m = |H\{g,g^{-1}\}H|/|H|$ is odd. In the following we will prove that $gH$ contains an involution. Once this is achieved, it then follows from Theorem \ref{basic} that $H$ is a perfect code of $G$. 

Since $(HgH)^{-1} = Hg^{-1}H$, $HgH$ and $Hg^{-1}H$ have the same cardinality and hence contain the same number of left cosets of $H$ in $G$. Thus, if $HgH \cap Hg^{-1}H = \emptyset$, then $H\{g,g^{-1}\}H = HgH \cup Hg^{-1}H$ would be the union of an even number of distinct left cosets of $H$ in $G$, but this contradicts our assumption that $m$ is odd. So we have $HgH \cap Hg^{-1}H\neq\emptyset$, which implies $H\{g,g^{-1}\}H=HgH=Hg^{-1}H$. In particular,  $g^{-1}=h_{1}gh_2$ for some $h_1,h_2\in H$. Since $g\notin G\setminus H$ by our assumption, we have $gh_1\notin G\setminus H$. Since $(gh_1)^2=h_{2}^{-1}h_1\in H$, it follows that $gh_{1}$ is of even order. Let $s$ be the largest odd divisor of the order of $gh_{1}$. Set $x=(gh_1)^{s}$. Then $x$ is a $2$-element in $G\setminus H$ with $x^{2}\in H$. Since $x=gh_1(gh_1)^{s-1}$ and $h_1(gh_1)^{s-1}\in H$, we have $HxH=HgH$. Set $L=x^{-1}Hx\cap H$. Since $m$ is odd and $HxH$ is the union of $m$ left cosets of $H$ in $G$, it follows that $L$ is of odd index in $H$. Since $x^{-1}Lx=x^{-2}Hx^2\cap x^{-1}Hx=H\cap x^{-1}Hx=L$, $L$ is normal and is of index $2$ in $\langle x,L\rangle$. Let $P$ be a Sylow $2$-subgroup of $\langle x,L\rangle$ such that $x\in P$. Set $Q=P\cap L$. Then $x^{2}\in Q$ and $Q$ is of index $2$ in $P$. It follows that $Q$ is normal in $P$ and is a Sylow $2$-subgroup of $L$. Since $L$ is of odd index in $H$, $Q$ is also a Sylow $2$-subgroup of $L$.
By our assumption, $H$ has a Sylow $2$-subgroup which is a perfect code of $G$. On the other hand, by Sylow's Theorem (\cite[Theorem 3.2.3]{KS2004}), any two Sylow $2$-subgroups of $H$ are conjugate in $H$. Thus, by Lemma \ref{conjugate}, $Q$ is a perfect code of $G$. Since $Q$ is normal in $P$ and $x\in P$, we have $x\in N_{G}(Q)$. Hence, by Corollary \ref{equivalent2}, there exists $b\in Q$ such that $xb$ is an involution. Since $Q$ is a subgroup of $H$, we have $b\in H$. Since $xb = gh_{1}(gh_{1})^{s-1}b$ and $h_1(gh_{1})^{s-1}\in H$, the involution $xb$ is contained in $gH$, as required.
\end{proof}

Let $G$ be a group and $K$ a subgroup of $G$. If $K$ is of odd order, then it is called \cite{KS2004} a \emph{$2'$-subgroup} of $G$; if the order of $K$ is the largest odd divisor of the order of $G$, then $K$ is called \cite{KS2004} a \emph{Hall $2'$-subgroup} of $G$. Our last result in this section gives a sufficient condition for the product of a $2$-subgroup and a $2'$-subgroup to be a subgroup perfect code.

\begin{theorem}\label{QK}
Let $G$ be a group. Let $Q$ be a $2$-subgroup of $G$ and $K$ a $2'$-subgroup of $G$. Suppose that all Sylow $2$-subgroups of $N_{G}(Q)$ are contained in $N_{G}(K)$. Then $QK$ is a perfect code of $G$ if and only if $Q$ is a perfect code of $G$.

In particular, if $Q$ is a $2$-subgroup of $G$ and $K$ is a normal $2'$-subgroup of $G$, then $QK$ is a perfect code of $G$ if and only if $Q$ is a perfect code of $G$.
\end{theorem}

\begin{proof}
Since all Sylow $2$-subgroups of $N_{G}(Q)$ are contained in $N_{G}(K)$, $Q$ is contained in $N_{G}(K)$. Hence $QK$ is a subgroup of $G$ and $K$ is normal in $QK$. Since the order of $Q$ is a power of $2$ but the order of $K$ is odd, we have $Q\cap K=\{1\}$. Therefore, $Q$ is a Sylow $2$-subgroup of $QK$ and $K$ is a Hall $2'$-subgroup of $QK$. Thus, by Theorem \ref{ns}, if $Q$ is a perfect code of $G$, then $QK$ is a perfect code of $G$. This proves the sufficiency.

We now prove the necessity. Assume that $H=QK$ is a perfect code of $G$. Consider an arbitrary $2$-element $x \in N_{G}(Q)$ with $x^{2}\in Q$. Then $Q$ is contained in $H\cap x^{-1}Hx$. Since $x^{2}\in H$, we have $H\{x,x^{-1}\}H=HxH$. Since $Q$ is a Sylow $2$-subgroup of $H$ and is contained in $H\cap x^{-1}Hx$, $H\cap x^{-1}Hx$ is of odd index in $H$. It follows that $|H\{x,x^{-1}\}H|/|H|$ is odd. By Theorem \ref{basic}, there exists
$ab\in H$ where $a\in Q$ and $b\in K$ such that $(xab)^{2}=1$. Since $x$ is a $2$-element of $N_{G}(Q)$, $x$ is contained in a Sylow $2$-subgroup of $N_{G}(Q)$. Since all Sylow $2$-subgroups of $N_{G}(Q)$ are contained in $N_{G}(K)$, we have $x\in N_{G}(K)$. Therefore, $(xa)^{2}=xab^{-1}a^{-1}x^{-1}b^{-1}\in K$. On the other hand, $(xa)^{2}=x^{2}x^{-1}axa\in Q$. Since $Q\cap K=\{1\}$, it follows that $(xa)^{2}=1$. Thus, by Corollary \ref{equivalent2}, $Q$ is a perfect code of $G$.
\end{proof}

\section{Subgroup perfect codes in a few classes of groups}
\label{sec:classes}

In this section we study subgroup perfect codes in a few classes of groups, namely metabelian groups, generalized dihedral groups, nilpotent groups and $2$-groups.

The following result gives a necessary and sufficiency condition for a normal subgroup of a metabelian group to be a subgroup perfect code.

\begin{theorem}
\label{ma}
Let $G$ be a metabelian group and $H$ a normal subgroup of $G$. Then $H$ is a perfect code of $G$ if and only if it has a Sylow $2$-subgroup which is a perfect code of $G$.
\end{theorem}

\begin{proof}
The sufficiency follows from Theorem \ref{ns} immediately. To prove the necessity, we assume that $H$ is a perfect code of $G$. Let $Q$ be a Sylow $2$-subgroup of $H$ and set $B=Q[G,H]$. Since by our assumption $H$ is a normal subgroup of $G$, $[G,H]$ is contained in $H$. Therefore $B$ is a subgroup of $H$. Since $Q$ is contained in $B$ and is a Sylow $2$-subgroup of $H$, the index of $B$ in $H$ is odd. It follows from Theorem \ref{commutator} that $B$ is a perfect code of $G$.
Since $G$ is a metabelian group and $H$ is normal in $G$, $[G,H]$ is a normal Abelian subgroup of $G$. It follows that the Hall $2'$-subgroup of $[G,H]$ is a characteristic subgroup of $[G,H]$ and therefore a normal subgroup of $G$. Note that the Hall $2'$-subgroup of $[G,H]$ is also the Hall $2'$-subgroup of $B$. Therefore, by Theorem \ref{QK}, $Q$ is a perfect code of $G$.
\end{proof}

A generalized dihedral group \cite{KS2004} is a group of the form $G=A\rtimes \langle b\rangle$, where $A$ is a normal Abelian subgroup of $G$ and $b$ is an involution satisfying $b^{-1}ab=a^{-1}$ for every $a\in A$. The next result gives a necessary and sufficient condition for a subgroup of a generalized dihedral group to be a perfect code of the group. It asserts that a subgroup of $G$ is a perfect code of $G$ if and only if either it is not contained in $A$ or is a subgroup perfect code of $A$. In the special case when $G$ is a dihedral group, this result gives \cite[Theorem 2.11(a)]{HXZ18}.

\begin{theorem}
\label{gen-dih}
Let $G=A\rtimes \langle b\rangle$ be a generalized dihedral group. Then a subgroup of $G$ is a perfect code of $G$ if and only if either it is not a subgroup of $A$ or it is a subgroup perfect code of $A$.
\end{theorem}

\begin{proof}
Suppose that $H$ is a perfect code of $G$. If $H$ is a subgroup of $A$, then by Lemma \ref{sub}, $H$ is a perfect code of $A$.

We now prove the sufficiency.

\smallskip
\textbf{Case 1.}~$H$ is not a subgroup of $A$.

In this case we have $ab\in H$ for some $a\in A$. If $H=\langle ab\rangle$, then $A$ is a Cayley transversal of $H$, and so $H$ is a perfect code of $G$ by Lemma \ref{Ct}. Assume that $H\neq\langle ab\rangle$. Then $H$ contains at least one element of $A$. Let $c$ be such an element. Then for any $g\in G\setminus H$, $gab$ or $gc$ is an involution. In other words, the coset $gH$ contains at least one involution. By Theorem \ref{basic}, we conclude that $H$ is a perfect code of $G$.

\smallskip
\textbf{Case 2.}~$H$ is a subgroup perfect code of $A$.

Note that $H$ is normal in $G$. Consider an arbitrary element $x\in G$ with $x^{2}\in H$. If $x^{2}\neq 1$, then $x\in A$. Since $H$ is a perfect code of $A$, by Lemma \ref{HXZ}, there exists $h\in H$ such that $(xh)^{2}=1$. Thus, by Lemma \ref{HXZ} again, $H$ is a perfect code of $G$.
\end{proof}

The next result shows that for nilpotent groups the problem of determining whether a subgroup is a perfect code can be reduced to the one of determining whether a Sylow $2$-subgroup is a perfect code.

\begin{theorem}
\label{nilp}
Let $G$ be a nilpotent group and $H$ a subgroup of $G$. Then $H$ is a perfect code of $G$ if and only if the Sylow $2$-subgroup of $H$ is a perfect code of $G$.
\end{theorem}

\begin{proof}
Since $G$ is nilpotent and $H$ is a subgroup of $G$, $H$ is nilpotent. Let $Q$ and $K$ be the Sylow $2$-subgroup of $G$ and the Hall $2'$-subgroup of $G$, respectively. Then $H=QK$. Since the Sylow $2$-subgroup of $G$ is contained in $N_{G}(K)$, the Sylow $2$-subgroup of $N_{G}(Q)$ is contained in $N_{G}(K)$. Thus, by Theorem \ref{QK}, $H$ is a perfect code of $G$ if and only if $Q$ is a perfect code of $G$.
\end{proof}

Recall that the Frattini subgroup \cite{KS2004} $\Phi(G)$ of a group $G$ is the intersection of all maximal subgroups of $G$. Equivalently, $\Phi(G)$ is the set of elements $g$ of $G$ with the property that $G = \langle g, X \rangle$ always implies $G = \langle X \rangle$ when $X$ is a subset of $G$.

\begin{theorem}
\label{2gp}
Let $G$ be a $2$-group. Then either each cyclic subgroup generated by an element of $G \setminus \Phi(G)$ is a perfect code or there exists a generalized quaternion subgroup of $G$ which is a perfect code of $G$.
\end{theorem}

\begin{proof}
Suppose that there exists $c \in G \setminus \Phi(G)$ such that $C := \langle c\rangle$ is not a perfect code of $G$. Then $C$ is a proper subgroup of $G$. By  Corollary \ref{equivalent}, there exists $b\in N_{G}(C)$ with $b^{2}\in C$ but $bC$ contains no involution. Set $H=\langle b,c\rangle$. Then $H=C\cup bC$ and therefore $H$ contains exactly one involution. Since $c \in G \setminus \Phi(G)$ and $b\notin C$, $H$ is not cyclic. It is known that a noncyclic 2-group which contains exactly one involution must be a generalized quaternion group (see \cite[Theorem 5.3.7]{KS2004}). Hence $H$ is a generalized quaternion group. If $G=H$, then $H$ is a perfect code of $G$. In the rest of the proof we assume that $G\neq H$.  Consider an arbitrary element $x\in N_{G}(H)$ with $x^{2}\in H$ and set $L=\langle x, H\rangle$. If $x\in H$, then $x^{-1}\in H$ and $(xh)^{2}=1$ for $h = x^{-1}$. If $x\notin H$, then $L$ is not a generalized quaternion group and hence contains at least two involutions. Since $L=H\cup xH$ and $H$ contains exactly one involution, there exists $h\in H$ such that $(xh)^{2}=1$. Thus, by Corollary \ref{equivalent}, $H$ is a perfect code of $G$.
\end{proof}

\smallskip

\noindent {\textbf{Acknowledgements}}~~We are grateful to the two anonymous referees whose comments and suggestions led to significant improvements of this paper. The first author was supported by the National Natural Science Foundation of China (No.~11671276), the Basic Research and Frontier Exploration Project of Chongqing (No.~cstc2018jcyjAX0010), and the Science and Technology Research Program of Chongqing Municipal Education Commission (No.~KJQN201800512). The second author was supported by the National Natural Science Foundation of China (No.~61771019) and the Research Grant Support Scheme of The University of Melbourne.




\begin{thebibliography}{00}


\bibitem{B77}
E. Bannai, On perfect codes in the Hamming schemes $H(n,q)$ with $q$ arbitrary,
{\em J. Combin. Theory Ser. A} 23 (1977) 52--67.

\bibitem{B77a}
E. Bannai, Codes in bipartite distance-regular graphs,
{\em J. London Math. Soc. (2)} 16 (1977) 197--202.

\bibitem{Big}
N. L. Biggs, Perfect codes in graphs, {\em J. Combin. Theory Ser. B} 15 (1973) 289--296.

\bibitem{BCN89}
A. E. Brouwer,  A. M. Cohen, and A. Neumaier, \emph{Distance-regular Graphs}, Springer-Verlag, Berlin, 1989.

\bibitem{DeS}
I. J. Dejter and O. Serra, Efficient dominating sets in Cayley graphs, {\em Discrete Appl. Math.} 129 (2003) 319--328.

\bibitem{Del}
P. Delsarte, An algebraic approach to the association schemes of coding theory, {\em Philips Res. Rep. Suppl.}, 10, 1973.

\bibitem{DSLW16}
Y-P. Deng, Y-Q. Sun, Q. Liu and H.-C. Wang, Efficient dominating sets in circulant graphs, {\em Discrete Math.} 340 (2017) 1503--1507.

\bibitem{Dinitz06}
M. Dinitz, Full rank tilings of finite abelian groups, {\em SIAM J. Discrete Math.} 20 (2006) 160--170.

\bibitem{FHZ}
R. Feng, H. Huang, and S. Zhou, Perfect codes in circulant graphs, {\em Discret. Math.} 340 (2017) 1522--1527.

\bibitem{HS}
P. Hammond and D. H. Smith, Perfect codes in the graphs $O_k$, {\em J. Combin. Theory Ser. B} 19 (1975) 239--255.

\bibitem{Hajos}
G. Haj\'{o}s, \"{U}ber einfache und mehrfache Bedeckung des $n$-dimensionalen Raumes mit einem W\"{u}rfelgitter, \textit{Math. Z.} 47 (1942) 427--467.

\bibitem{HHS}
T. W. Haynes, S. T. Hedetniemi and P. Slater, \textit{Fundamentals of Domination in Graphs}, Marcel Dekker, Inc., New York, 1998.

\bibitem{Heden1}
O. Heden, A survey of perfect codes, {\em Adv. Math. Commun.} 2 (2008) 223--247.

\bibitem{HK18}
P. Horak and D. Kim, 50 years of the Golomb-Welch conjecture, {\em IEEE Trans. Inform. Theory} 64 (2018) 3048--3061.

\bibitem{HXZ18}
H. Huang, B. Xia, and S. Zhou, Perfect codes in Cayley graphs, {\em SIAM J. Discrete Math.} 32 (2018) 548--559.

\bibitem{Kr86}
J. Kratochv\'{i}l, Perfect codes over graphs, {\em J. Combin. Theory Ser. B} 40 (1986) 224--228.

\bibitem{KS2004}
H. Kurzweil and B. Stellmacher, \emph{The Theory of Finite Groups, An Introduction},
Universitext, Springer, New York-Berlin-Heidelberg, 2004.

\bibitem{Le}
J. Lee, Independent perfect domination sets in Cayley graphs, {\em J. Graph Theory} 37 (2001) 213--219.

\bibitem{Lens}
H. W. Lenstra, Jr.,  Two theorems on perfect codes, {\em Discrete Math.} 3 (1972) 125--132.

\bibitem{Va75}
J. H. van Lint, A survey of perfect codes, {\em Rocky Mountain J. Math.} 5 (1975) 199--224.

\bibitem{MWWZ19}
X. Ma, G. L. Walls, K. Wang, and S. Zhou, Subgroup perfect codes in Cayley graphs, submitted, \url{https://arxiv.org/abs/1904.01858}.

\bibitem{MS77}
F. J. MacWilliams and N. J. A. Sloane, \emph{The Theory of Error-correcting Codes},
North-Holland, Amsterdam, 1977.

\bibitem{Szabo06}
S. Szab\'{o}, Factoring finite abelian groups by subsets with maximal span,
\emph{SIAM J. Discrete Math.} 20 (2006) 920--931.

\bibitem{SS}
S. Szab\'{o} and A. Sands, \emph{Factoring Groups into Subsets}, CRC Press, Boca Raton, FL, 2009.

\bibitem{Ta13}
T. Tamizh Chelvam and S. Mutharasu, Subgroups as efficient dominating sets in Cayley
graphs, {\em Discrete Appl. Math.} 161 (2013) 1187--1190.

\bibitem{Zhou2016}
S. Zhou, Total perfect codes in Cayley graphs, \textit{Des. Codes Cryptogr.} 81 (2016) 489--504.

\bibitem{Z15}
S. Zhou, Cyclotomic graphs and perfect codes, {\em J. Pure Appl. Algebra} 223 (2019) 931--947.

\end{thebibliography}


\end{document}